\documentclass[12pt]{amsart}
\usepackage{amsfonts,amssymb}
\usepackage[all,ps,cmtip]{xy}

\usepackage{mathrsfs}
\let\mathcal\mathscr

\usepackage{array}

\usepackage[all,ps,cmtip]{xy}

\def\Z{{\bf Z}}

\def\C{{\bf C}}

\def\B{{\bf B}}
\def\R{{\bf R}}
\def\Q{{\bf Q}}
\def\P{{\bf P}}

\def\cO{\mathcal{O}}

\def\cC{\mathcal{C}}

\def\cZ{\mathcal{Z}}

\def\llra{\hbox to 10mm{\rightarrowfill}}
\def\lllra{\hbox to 15mm{\rightarrowfill}}

\def\llla{\hbox to 10mm{\leftarrowfill}}
\def\lllla{\hbox to 15mm{\leftarrowfill}}

\def\hra{\hookrightarrow}

\def\eps{\varepsilon}
\def\ie{\hbox{i.e.}}

\DeclareMathOperator{\isomto}{\stackrel{{}_{\scriptstyle\sim}}{\to}}

\DeclareMathOperator{\mult}{mult}

\def\Im{\mathop{\rm Im}\nolimits}

\DeclareMathOperator{\Ker}{Ker}

\DeclareMathOperator{\pr}{pr}

\def\Re{\mathop{\rm Re}\nolimits}

\DeclareMathOperator{\Hdg}{Hdg}
\DeclareMathOperator{\Supp}{Supp}

\DeclareMathOperator{\innt}{int}

\def\llra{\hbox to 10mm{\rightarrowfill}}
\def\lllra{\hbox to 15mm{\rightarrowfill}}

\newtheorem{lemm}{Lemma}[section]
\newtheorem{theo}[lemm]{Theorem}
\newtheorem{coro}[lemm]{Corollary}
\newtheorem{prop}[lemm]{Proposition}
\newtheorem{conj}[lemm]{Conjecture}

\theoremstyle{definition}

\newtheorem{rema}[lemm]{Remark}

\theoremstyle{remark}
\newtheorem*{remark*}{Remark}
\newtheorem*{note*}{Note}

\begin{document}
\title{Pseudo-effective classes and pushforwards}

 \author[O. Debarre]{Olivier Debarre}
\address{D\'epartement de  Math\'ematiques et Applications\\
CNRS UMR 8553\\
\'Ecole Normale Sup\'erieure\\
45 rue d'Ulm\\
75230 Paris cedex 05, France}
\email{{\tt olivier.debarre@ens.fr}}
\author[Z. Jiang]{Zhi Jiang}
\address{D\'epartement de Math\'ematiques\\
CNRS UMR 8628\\Universit\'e Paris-Sud\\B\^atiment 425\\91405 Orsay cedex, France}
\email{{\tt zhi.jiang@math.u-psud.fr}}
\author[C. Voisin]{Claire Voisin}
\address{Centre de math\'ematiques Laurent Schwartz\\CNRS UMR 7640\\\'Ecole Polytechnique\\91128 Palaiseau cedex, France }
\email{{\tt voisin@math.polytechnique.fr}}

 \begin{abstract}
 Given a morphism between complex projective varieties, we make several conjectures on the relations between the set of pseudo-effective (co)homology classes which are annihilated by pushforward and the set of classes of varieties contracted by the morphism. We prove these conjectures for classes of curves or divisors. We also prove that one of these conjectures implies Grothendieck's generalized Hodge conjecture for varieties with Hodge coniveau at least 1.
  \end{abstract}

 \maketitle

\begin{flushright} {\it To the memory of Andrey Todorov}
\end{flushright}

\section{Introduction and conjectures}

Let $X$ and $Y$ be    complex projective varieties and let $\phi:X\to Y$
be a  morphism. We say that an irreducible subvariety $Z$ of $X$ is {\em contracted by $\phi$} if $\dim (\phi(Z))<\dim (Z)$. If $k:=\dim(Z)$, this is equivalent to saying that the  class $[Z]\in H_{2k}(X,\R)$ is in the kernel of the push-forward morphism
$$ \phi_*:H_{2k}(X,\R )\rightarrow H_{2k}(Y,\R ).$$

In this article, we discuss    whether some extension of this equivalence to more general
 cycle classes is true. In $H_{2k}(X,\R)$,   classes of irreducible subvarieties of $X$ generate a closed convex
cone, called the pseudo-effective cone, and whose elements are called {\em pseudo-effective classes.} 
We are interested in  pseudo-effective classes annihilated by $\phi_*$ and make the following conjecture.

\begin{conj}\label{conj1} Let $X$ and $Y$ be  complex projective varieties and let $\phi:X\rightarrow Y$
be a  morphism.   Then any pseudo-effective   class   $ \alpha\in
H_{2k}(X,\R )$ such that $\phi_*\alpha=0$ in $H_{2k}(Y,\R )$ belongs
to the real vector space (resp.\ the closed convex cone) spanned by
classes of  $k$-dimensional subvarieties of $X$ which are contracted by
$\phi$.
\end{conj}

The ``resp.'' form of the conjecture is obviously stronger and will be refered to as the ``Strong Conjecture \ref{conj1},'' as opposed to the ``Weak Conjecture \ref{conj1}.''

Now assume  moreover that $X$ and $Y$ are {\em smooth} and  set $c :=\dim (X)-\dim (Y)$. Using   Poincar\'e duality, one     defines    Gysin morphisms
$$ \phi_*:H^{2l}(X,\R )\rightarrow H^{2l-2c}(Y,\R ) $$
 as the compositions
$$   H^{2l}(X,\R )\xrightarrow{PD_X} H_{2\dim(X)-2l}(X,\R )\xrightarrow{\phi_*} H_{2\dim(X)-2l}(Y,\R )\xrightarrow{PD_Y} H^{2l-2c}(Y,\R ).$$
Any subvariety $Z$ of $X$ has a cohomology class, and $Z$ is contracted by $\phi$ if and only if this class is annihilated by  the Gysin morphism. Conjecture \ref{conj1}    then takes the following form.

\begin{conj} \label{conj2}  Let $X$ and $Y$ be  smooth complex projective varieties and let $\phi:X\rightarrow Y$
be a  morphism.   Then any pseudo-effective   class   $ \alpha\in H^{2l}(X,\R )$
such that $\phi_*\alpha=0$
belongs to the real vector space (resp.\ the closed convex cone) spanned by classes of   $l$-codimensional subvarieties of $X$ which are contracted by $\phi$.
\end{conj}

Finally, a variant of Conjecture \ref{conj2}  is obtained by
modifying the positivity notion used: instead of algebraic
pseudo-effectivity, we can consider weak, that is, analytic
pseudo-effectivity, where we say that an algebraic cycle is
{\em weakly pseudo-effective} if its cohomology class is the class
of a positive current. Pseudo-effective cycle classes are weakly
pseudo-effective, but the converse is unknown except in the divisor
case (see \cite{demailly1}). More generally, we will say that a real
cohomology  class is weakly pseudo-effective  if it is a limit of
weakly pseudo-effective cycle classes.

We can consider the following   strengthening of the weak form of Conjecture \ref{conj2}.

\begin{conj}\label{conj5}  Let $X$ and $Y$ be  smooth complex projective varieties and let $\phi:X\rightarrow Y$
be a  morphism.   Then any weakly pseudo-effective   class   $ \alpha\in H^{2l}(X,\R )$
such that $\phi_*\alpha=0$
belongs to the real vector space spanned by classes of  $l$-codimensional subvarieties of $X$ which are contracted by $\phi$.
\end{conj}

Evidence for these conjectures is provided in section \ref{sf}, where we prove them when  no subvarieties are contracted by $\phi$ (\ie, when $\phi$ is finite). In section \ref{ss}, we prove that we may assume that $\phi$ is surjective and $Y$ is a projective space.  We   show furthermore that Conjecture \ref{conj1} can be
reduced to the case where $\phi$ has connected fibers (a reduction
that we were unable to perform in the smooth setting).

Our first main result is then the following.

\begin{theo} For curves ($k=1$ and $l=\dim(X)-1$), the Weak Conjecture \ref{conj1}, the Strong
Conjecture \ref{conj2}, and Conjecture \ref{conj5} all hold true.

For divisors ($l= 1$),   the Strong
Conjecture \ref{conj2} holds true.
\end{theo}

In other words, if $\phi:X\to  Y$ is a morphism
between projective varieties,  a pseudo-effective degree-2 homology class
annihilated by $\phi_*$ belongs to the vector space generated by
classes of  contracted curves. If $X$ is smooth, it belongs to the closed convex cone generated by these classes and the same result holds for divisor classes.

The proof of this theorem occupies sections \ref{seccurves} (for the curve case) and
\ref{secdivisors} (for the divisor case).

Finally, we  show in section \ref{secpsefan}
that  {\em Conjecture \ref{conj5}   implies a strong form of the generalized Hodge conjecture
  for
  smooth projective varieties $X$  such that $H^{i,0}(X)=0$  for $i>0$} (see Theorem \ref{th65}).
 This is similar  to the implication proved in  \cite{voi}, where effectivity of certain strongly positive cycles
on Fano variety of lines of complete intersections is proved
to imply the generalized Hodge conjecture for coniveau-2 complete intersections in projective spaces.

\subsection*{Acknowledgements} We thank Mihai Fulger for pointing out an error in a previous version of this article.

\section{The finite case}\label{sf}

A first evidence for our conjectures is provided by the case where
no subvarieties are contracted by $\phi$.

\begin{prop}\label{profinising}
Let $\phi: X\rightarrow Y$ be a
finite morphism between projective varieties. If $\alpha\in H_{2k}(X,\R) $ is pseudo-effective and
$\phi_*\alpha=0$ in $H_{2k}(Y,\R)$, then $\alpha=0$.\end{prop}

\begin{proof} Let $L$ be an ample line bundle on $Y$. Then $\phi^*L$ is ample on $X$. Choose a projective embedding
$i:X\hookrightarrow \P ^N$ such that
$ i^*\cO _{\P ^N}(1)= \phi^*L^{\otimes m}$  for some $m>0$, and let $\omega\in
A^{1,1}(\P^N,\R)$ be a positive Chern form of
$ \cO _{\P ^N}(1)$. As the class $[\omega]$ of
$\omega$ is $c_1(\cO _{\P ^N}(1))$, its restriction to
$X$ is $c_1(\phi^*( L^{\otimes m}))$, hence we have
$$\langle \alpha,[\omega]\vert^k_X\rangle_X=\langle\phi_*\alpha,c_1(L ^{\otimes m})^k\rangle_Y=0.$$
On the other hand, we have
$$\langle\alpha,[\omega]\vert^k_X\rangle_X=\langle i_*\alpha,[\omega]^k\rangle_{\P^N}.$$ Note that $\alpha$
being pseudo-effective, $i_*\alpha$ is pseudo-effective, and thus
$i_*\alpha$ is the class of a positive current $T$ of type $(k,k)$.
We now use the following standard result.

\begin{lemm}\label{leposcurrent} Let $T$ be a compactly supported positive current of type $(k,k)$ on a
  complex manifold endowed with a positive $(1,1)$-form
$\omega$. If $\int_T\omega^k=0$, then $T=0$.
\end{lemm}

We apply this lemma to a neighborhood $U$ of $X$ in $\P ^N$
which is a deformation retract of $X$. Let us denote by $i_U$ the
inclusion of $X$ in $U$. The homology class $i_{U*}\alpha\in
H_{2k}(U,\R)$ is the homology class of the compactly supported closed positive
current $T$ on $U$, hence it is $0$ by the lemma above. On the other
hand,     $i_{U*}:H_\bullet(X,\R)\isomto
H_ \bullet(U,\R)$ is an isomorphism, because $U$ retracts onto $X$. Thus $\alpha=0$ in
$H_{2k}(X,\R)$.
\end{proof}

\begin{rema}
The same proof   applies in the setting of {\em Conjecture} \ref{conj5}, which   {\em therefore holds true for finite morphisms.}
\end{rema}

We use Proposition \ref{profinising} to show the following.

\begin{coro} \label{profinsing} In order to prove Conjecture
\ref{conj1}, it suffices to prove it for morphisms with connected
fibers.
\end{coro}

\begin{proof}  Let
$\phi:X\rightarrow Y$ be a morphism, where $X$ and $Y$ are
projective, and let $\alpha\in H_{2k}(X,\R)$ be a
pseudo-effective class such that $\phi_*\alpha=0$. We use the Stein factorization
$\phi =h\circ\psi  $, where $\psi :X\rightarrow Z$ has connected
fibers and $h:Z\rightarrow Y$ is finite. The class
$\psi_*\alpha$ is pseudo-effective on $Z$, and it is annihilated by
$h_*$. By Proposition  \ref{profinising}, $\psi_*\alpha=0$. If Conjecture \ref{conj1} is true for
morphisms with connected
fibers, $\alpha$ is in the real vector space (resp.\ in the closed convex cone) spanned by classes
of subvarieties of $X$ contracted  by $\psi$, hence by $\phi$.
\end{proof}

   \section{Reduction to the surjective case}\label{ss}

\begin{prop} \label{lemma1} In order to prove any of the conjectures stated in the introduction, it suffices to do so when   $\phi$ is  surjective and $Y$ is a projective space.
\end{prop}

\begin{proof} Let us do it for Conjecture  \ref{conj2}. For the other two conjectures, the proof  is similar.
Set $Y':=  \phi(X)\subset Y$ and $d:=\dim (Y')$. Let us denote by
$\phi':X\rightarrow Y'$ the surjective morphism induced by $\phi$.
We choose a rational map  $ Y\dashrightarrow \P^d$ which restricts to a {\em finite morphism}
$\pi: Y'\to \P^d$. Such a map is obtained by restricting to $Y\subset \P^N$ a general linear projection $\P^N\dashrightarrow \P^d$.
The morphism $\pi\circ \phi$ is surjective and   contracts the same subvarieties as $\phi$.
Set $n:=\dim (X)$ and $c :=\dim (X)-\dim (Y)$. If $\phi_* \alpha =0$ in $H^{2l-2c}(Y,\R )$, then
for any line bundle $L$ on $Y$, we have
\begin{equation*}\label{crochet1}
\langle \phi_*\alpha, c_1(L)^{n-l}\rangle_Y=0,
\end{equation*}
 which implies
\begin{equation}\label{crochet2}
\langle\phi'_*\alpha, c_1(L\vert_{  Y'})^{n-l}\rangle_{Y'}=0.
\end{equation}
The brackets in the above equations  are computed as follows:
 the class $\phi_*\alpha \in H^{2l-2c}(Y,\R )$ is, by definition of the Gysin morphism
 $\phi_*$, the Poincar\'{e} dual on $Y$ of the homology class
 $\phi_*(PD_X(\alpha))\in H_{2n-2l}(Y,\R )$, where $PD_X(\alpha)\in H_{2n-2l}(X,\R )$ is the Poincar\'{e} dual on $X$
 of the cohomology class $\alpha$.
 The intersection number $\langle\phi_*\alpha, c_1(L)^{n-l}\rangle_Y$ is then given by the pairing
 between $H_{2n-2l}(Y,\R )$ and $H^{2n-2l}(Y,\R )$. 
 Similarly, the intersection number $\langle\phi'_*\alpha, c_1(L\vert_{  Y'})^{n-l}\rangle_{Y'}$ is  given by the pairing
 between $H_{2n-2l}(Y',\R )$ and $H^{2n-2l}(Y',\R )$.

 We apply (\ref{crochet2}) to the line bundle $L$ giving the rational projection $Y\dashrightarrow \P^d$. On $Y'$, we   have $L\vert_{Y'}=\pi^*\cO_{\P^d}(1)$, and thus we get from the projection formula
 \begin{equation}\label{crochet3}
0= \langle\phi'_*\alpha, c_1(L\vert_{  Y'})^{n-l}\rangle_{Y'}=
\langle(\pi\circ\phi')_*\alpha, c_1(\cO_{\P^d}(1))^{n-l}\rangle_{\P^d} ,
\end{equation}
where the second intersection number is   computed using the pairing between
$H_{2n-2l}(\P^d,\R )$ and $H^{2n-2l}(\P^d,\R )$.
Formula (\ref{crochet3}) says that $(\pi\circ\phi')_*(PD_X(\alpha))=0$ in
$H_{2n-2l}(\P^d,\R )$, hence $(\pi\circ\phi')_*  \alpha =0$ in
$H^{2d-(2n-2l)}(\P^d,\R )$.
The morphism $\pi\circ \phi'$ is surjective and   $X$ and $\P^d$ are both
smooth.

Hence,
if we know Conjecture \ref{conj2}  when the morphism is surjective onto a projective space,
we conclude that the class $\alpha$ belongs to the
real vector space (resp.\ to the closed convex cone) spanned by  classes of subvarieties contracted by $\pi\circ \phi'$, which
is the same as the real vector space (resp.\ to the closed convex cone) spanned by classes of subvarieties  contracted by $ \phi$.
  \end{proof}

\section{The  curve case}\label{seccurves}

The following theorem solves  the Strong Conjecture
\ref{conj2} in the curve case. It is a straightforward application of the relative Kleiman  criterion for ampleness.

\begin{theo} \label{curvecase} Let $X$ and $Y$ be  smooth complex projective varieties  and let $\phi:X\rightarrow Y$
be a  morphism. Set $n:=\dim(X)$. Then any pseudo-effective  class $\alpha\in H^{2n-2}(X,\R )$
such that  $ \phi_* \alpha =0$
belongs to the closed convex cone spanned by classes of curves on $X$ which are contracted by $\phi$.
\end{theo}

Before beginning the proof, we introduce some notation. For a smooth projective variety $X$, we denote by $N^1(X)$ its N\'eron-Severi group, \ie, the free abelian group of divisors on $X$ modulo numerical equivalence. By the Lefschetz $(1,1)$-theorem, the real vector space $N^1(X)_\R:=N^1(X)\otimes\R$ is isomorphic to $ H^{1,1}(X)_\R$.

Similarly, we denote by $N_1(X)$   the group of 1-cycles on $X$ modulo numerical equivalence. By the Lefschetz $(1,1)$-theorem and the hard Lefschetz theorem, the real vector space $N_1(X)_\R:=N_1(X)\otimes\R$ is isomorphic to $ H^{n-1,n-1}(X)_\R$ and it is dual to $N^1(X)_\R$. It contains the closed convex cone spanned by classes of curves in $X$, called the Mori cone and denoted by $\overline{NE}(X)$.

\begin{proof}[Proof of the theorem] In the vector space $N_1(X)_\R$, we have two closed convex cones: the intersection   $\cC_1=\overline{NE}(X)\cap \Ker(\phi_*)$ and  the closed convex cone $\cC_2$ spanned by classes of curves contracted by $\phi$. We  have $\cC_2\subset \cC_1$ and we need to prove that these cones are equal.

We argue by contradiction. If $\cC_2\subsetneq \cC_1$, there exists a class $d\in N^1(X)_\R=N_1(X)_\R^\vee$ such that $d$ is positive on $\cC_2 \setminus\{0\}$, but $d\cdot \alpha<0$ for some $\alpha\in \cC_1$. We may even  choose for $d$ the class of a divisor $D$ on $X$. The divisor $D$ is then $\phi$-ample by  the relative Kleiman  criterion (\cite[Theorem 1.44]{km}). If $H$ is an ample divisor on $Y$,   the divisor $mp^*H +D$ is therefore ample for $m\gg 0$. But we have
$$(mp^*H +D)\cdot \alpha =m( H  \cdot \phi_*\alpha)+  (d \cdot \alpha) <0,$$
 which, by the usual Kleiman  criterion, contradicts $\alpha\in \overline{NE}(X)$.
  \end{proof}

We now prove  Conjecture \ref{conj1} in its weak form for
  degree-2  homology classes. Note that the proof above  does not apply because we do not
have in the singular case
 a duality between the real vector   space spanned by homology classes of curves and the space $N^1(X)_\R$. However, the proof given below does apply in the setting of {\em Conjecture} \ref{conj5}, which   {\em therefore holds true in the curve case.}

\begin{theo} \label{theo5.4} Let $X$ and $Y$ be  complex projective varieties and let $\phi:X\rightarrow Y$
be a  morphism.   Then any pseudo-effective   class   $ \alpha\in
H_{2}(X,\R )$ such that $\phi_*\alpha=0$ in $H_{2}(Y,\R )$ belongs
to the real vector space  spanned by
classes of  curves in $X$ which are contracted by
$\phi$.
\end{theo}

\begin{proof} By Proposition \ref{lemma1},   we can assume that $Y$ is smooth. Next, we  prove the following result.

\begin{lemm} In order to prove Theorem \ref{theo5.4}, it suffices to show
that, under the same assumptions, the class $\alpha$ belongs to the real vector subspace $V$  of $H_2(X,\R)$ spanned by
the $j_{y*}(H_2(X_y,\R))$, for $y\in Y$, where $j_y:X_y\hra X$ is the inclusion
of the fiber $X_y$.
\end{lemm}

\begin{proof} This follows from mixed Hodge theory and is proved as follows.
The set of cycle classes in $H_2(X,\Q)$ is contained in the pure part $W_{\rm pure}H_2(X,\Q)$ of
the mixed Hodge structure on $H_2(X,\Q)$, since cycle classes lift to
cycle classes on a desingularization of $X$.
Choose finitely many points $y_1,\dots,y_N\in Y$ such that
$V$ is spanned by the spaces $j_{y_i*}(H_2(X_{y_i},\Q))$. The morphism
$$J:=\sum_ij_{y_i*}:\bigoplus_iH_2(X_{y_i},\Q)\rightarrow H_2(X,\Q)$$
is a morphism of mixed Hodge structures, hence
is strict for the weight filtration by \cite{deligneII}.
Since $\alpha\in (\Im (J)\otimes \R)\cap W_{\rm pure}H_2(X,\R)$, it belongs
to $J(\bigoplus_i W_{\rm pure} H_2(X_{y_i},\Q))\otimes \R$.
Furthermore, the morphism $J: \bigoplus_i W_{\rm pure} H_2(X_{y_i},\Q)\rightarrow  W_{\rm pure}H_2(X,\Q)$
is a morphism of polarized Hodge structures, hence maps
the space of Hodge classes $\Hdg_2(\bigoplus_i W_{\rm pure} H_2(X_{y_i},\Q))$
 onto the set of Hodge classes in $\Hdg_2(W_{\rm pure}H_2(X,\Q))$ contained in $\Im (J)$. If our class $\alpha$ belongs
 to $(\Hdg_2(W_{\rm pure}H_2(X,\Q))\cap \Im (J))\otimes \R$,  it thus belongs to the
 space
 $$\Big(\bigoplus_ij_{y_i*}\big(\Hdg_2(W_{\rm pure} H_2(X_{y_i},\Q))\big)\Big)\otimes \R\subset H_2(X,\R).$$
 To conclude, it suffices to recall that Hodge classes in $\Hdg_2(W_{\rm pure} H_2(X_{y_i},\Q))$
 are algebraic, because they come from degree-2 Hodge homology classes on a desingularization $\widetilde X_{y_i}$,
 which are   known to be algebraic on $\widetilde X_{y_i}$.
 \end{proof}

Going back to our pseudo-effective class $\alpha\in H_2(X,\R)$,
 we choose an imbedding  $i:X\hookrightarrow\P^N$ and denote by $U$ a neighborhood
 of the image of $(\phi,i)$ in $Y\times \P^N$. Thus $\phi$ is the composition of the inclusion $i'$
 of $X$ into $U$ and of the first projection $p :U\rightarrow Y$.
If $U$ is small enough, it retracts over $Y$ onto $X$, so that
we have an isomorphism $i'_*:H_2(X,\R)\isomto H_2(U,\R)$,
and for each $y\in Y$, we have an isomorphism
$$H_2(X_y,\R)\isomto H_2(U_y,\R)$$
compatible with the maps $j_{y*}$ and $j_{U_y*}$, where $j_{U_y}:U_y\hra U$ is the inclusion
of the fiber $U_y$.
It thus suffices to prove that
$i'_*\alpha\in H_2(U,\R)$ belongs to the subspace spanned by
the $j_{U_y*}(H_2(U_y,\R))$, for $y\in Y$.
 Since $U$ is smooth, we can use  currents to compute $H_2(U,\R)$. In fact,
 $\alpha$ being written as $\lim_k[C_k]_{\rm fund}$, where the $C_k$ are effective
 $1$-cycles on $X$ with $\Q$-coefficients,
 $i'_*\alpha$ is the class of the weakly positive current $T:=\lim_kT_{C_k}$, where $T_{C_k}$
 is the current of integration over $i'(C_k)$ acting on $2$-forms of $U$. The limiting current exists,
 upon taking a subsequence if necessary, by \cite[Chapter III, Proposition 1.23]{demailly}.
The current $T$ is positive, closed of type $(N+r-1,N+r-1)$,
compactly supported on $U$, and it satisfies
\begin{equation}\label{eqcurrentT}p_*T=0
\end{equation}
as a current of type $(r-1,r-1)$ on $Y$.
Indeed, the current $p_*T$ is positive on $Y$ and its homology class is equal to $\phi_*\alpha\in H_2(Y,\R)$.
The vanishing (\ref{eqcurrentT}) thus follows from Lemma \ref{leposcurrent}.
Choose a volume form $\nu$ on $Y$.
We claim that locally on $U$, we have
\begin{eqnarray}\label{eqcurrentT2} T=p^*\nu\wedge T'
\end{eqnarray}
for a current $T'$  of type $(N,N)$ supported on $X\subset U$. By (\ref{eqcurrentT2}),
 we mean that for any
$2$-form $\eta$ on $U$, one has $T(\eta)=T'(p^*\nu\wedge \eta)$.
To see this, we use the fact that $U$ is contained in the product
$Y\times \P^N$. Our local holomorphic coordinates on $U$ will be of the form
$(y_1,\ldots,y_r,z_1,\ldots,z_N)$ where $(y_1,\ldots,y_r)$ are
local holomorphic coordinates on $Y$ and $(z_1,\ldots,z_N)$ are
local holomorphic coordinates on $\P^N$.
Write
\begin{multline*}
T=\iota^{r+N-1}\sum_{i,j=1}^r f_{ij}\innt \left(\frac{\partial}{\partial y_i}\wedge \frac{\partial}{\partial \overline{y}_j} \right)( dy_1\wedge d\overline{y}_1\wedge\dots\wedge dy_r\wedge d\overline{y}_r\wedge dz_1\wedge d\overline{z}_1\wedge\dots\wedge dz_N\wedge d\overline{z}_N)\\
{}+\iota^{r+N-1}
\sum_{k,l=1}^N g_{kl} \innt \left(\frac{\partial}{\partial z_k}\wedge \frac{\partial}{\partial \overline{z}_l} \right)(dy_1\wedge d\overline{y}_1\wedge\dots\wedge dy_r\wedge d\overline{y}_r\wedge dz_1\wedge d\overline{z}_1\wedge\dots\wedge dz_N\wedge d\overline{z}_N) \\
{}+\iota^{r+N-1}\sum_{i\leq r,\ k\leq N} h_{ik}\innt \left(\frac{\partial}{\partial y_i}\wedge \frac{\partial}{\partial \overline{z}_k}\right)( dy_1\wedge d\overline{y}_1\wedge\dots\wedge dy_r\wedge d\overline{y}_r\wedge dz_1\wedge d\overline{z}_1\wedge\dots\wedge dz_N\wedge d\overline{z}_N)\\
{}+\iota^{r+N-1}\sum_{i\leq r,\ k\leq N} \overline{h_{ik}}\innt \left(\frac{\partial}{\partial \overline{y}_i}\wedge \frac{\partial}{\partial {z}_k} \right)( dy_1\wedge d\overline{y}_1\wedge\dots\wedge dy_r\wedge d\overline{y}_r\wedge dz_1\wedge d\overline{z}_1\wedge\dots\wedge dz_N\wedge d\overline{z}_N),
\end{multline*}
 where $\iota^2=-1$ and  the  coefficients are distributions. The positivity condition for $T$ says that
for any form $\alpha=\sum_i\alpha_idy_i+\sum_k\alpha_kdz_k$ of type $(1,0)$, one has  $T(\iota \alpha\wedge\overline{\alpha})\geq0$, or equivalently that the Hermitian matrix  written by blocks as
$$M=\begin{pmatrix} f_{ij}&h_{ik}\\
\overline{h_{ki}}&g_{kl}
\end{pmatrix}$$
is a semi-positive Hermitian matrix of distributions, meaning that for any $(r+N)$-uple $Z=(Z_r,Z_N)$ of complex functions,   with $Z_r=(\alpha_i)$ and $Z_N=(\alpha_k)$, $^t\overline{Z}MZ$ is a measure on $U$.

Since $p_*T=0$, we get that for any
form $\alpha=\sum_i\alpha_idy_i$ on $Y$, one has $T(p^*(\iota \alpha\wedge\overline{\alpha}))=0$. As this is the
integral over $U$ of the measure $^tZ_r(f_{ij})Z_r$, where $Z_r=(\alpha_i)$,
 we conclude that this measure is $0$,
that is $f_{ij}=0$ for all $i,j\in\{1,\dots,r\} $. The semi-positivity  of $M$ implies then
$(h_{ik})=0$, so that
$$T=\iota^{r+N-1}
 \sum_{k,l=1}^N g_{kl}\innt \left(\frac{\partial}{\partial z_k}\wedge \frac{\partial}{\partial \overline{z}_l} \right)(dy_1\wedge d\overline{y}_1\wedge\dots\wedge dy_r\wedge d\overline{y}_r\wedge dz_1\wedge d\overline{z}_1\wedge\dots\wedge dz_N\wedge d\overline{z}_N) .$$
Assuming that $\nu=\iota^r dy_1\wedge d\overline{y}_1\wedge\dots\wedge dy_r\wedge d\overline{y}_r$, this gives the desired result with
$$T'=\iota^{N-1}\sum_{k,l=1}^N g_{kl} \innt\left(\frac{\partial}{\partial z_k}\wedge \frac{\partial}{\partial \overline{z}_l} \right)( dz_1\wedge d\overline{z}_1\wedge\dots\wedge dz_N\wedge d\overline{z}_N).$$

\begin{lemm} \label{lefinalcurrent}The subspace $K:=\bigcap_{y\in Y}\Ker(H^2(U,\R)\rightarrow H^2(U_y,\R))$ is the set of de Rham cohomology classes of closed $2$-forms
on $U$ vanishing on all fibers $U_y$.
\end{lemm}

\begin{proof} Let $\alpha\in K$. Then for any $y\in Y$, there is an open subset
$V_y\subset Y$ containing $y$ such that $\alpha\vert_{  U_{V_y}}=0$ (we use for this the fact that for any $y\in Y$, there is a neighborhood
$W_y$ of $y$ in $Y$  such that $U_{W_y}$ has the same homotopy type as $U_y$). The class
$\alpha$ is represented by a closed $2$-form
$\tilde{\alpha}$ and we have
$$\tilde \alpha\vert_{ U_{V_y}}=d\beta_y.$$
Let $(f_i)$ be a partition of unity
on $Y$ relative to an open covering of $Y$ by finitely many of the $V_y$, say $V_{y_1},\ldots,\, V_{y_M}$.
Then $\tilde{\alpha}-d(\sum_i(p^*f_i)\beta_{y_i})$ has the same class as $\tilde{\alpha}$
and it is equal to
$$\tilde{\alpha}-\sum_ip^*f_id\beta_{y_i}-\sum_ip^*df_i\wedge \beta_{y_i}=-\sum_ip^*df_i\wedge \beta_{y_i},$$
which clearly vanishes on each fiber $U_y$.
\end{proof}

We now finish the proof of Theorem \ref{theo5.4}. We wanted to prove that the homology class of $T$
belongs to the space spanned by the $j_{U_y*}H_2(U_y,\R)$ for $y\in Y$.
The orthogonal of this space in $H^2(U,\R)$ is
$K$. The result  is thus equivalent to saying
that for any $\eta\in K$, one has $T(\eta)=0$.
By Lemma \ref{lefinalcurrent}, $\eta$ is the class of a closed form $\widetilde{\eta}$ vanishing on $U_y$  for any $y$.
On the other hand, it follows  from (\ref{eqcurrentT2}) that for any $2$-form
$\widetilde{\eta}$ on $U$ vanishing on the fibers $U_y$,  we have $T(\widetilde{\eta})=0$.
\end{proof}

\section{The  divisor case\label{secdivisors}}

\subsection{The Weak Conjecture \ref{conj2} for divisors}

The following theorem solves Conjecture
\ref{conj2} in its weak form in the divisor case.

\begin{theo} \label{theodiv} Let $X$ and $Y$ be  smooth complex projective varieties  and let $\phi:X\rightarrow Y$
be a  morphism. Then any pseudo-effective  class $\alpha\in H^2(X,\R )$
such that  $ \phi_* \alpha =0$
belongs to the real vector space spanned by classes of divisors on $X$ which are contracted by $\phi$.
\end{theo}

\begin{proof} Using Proposition \ref{lemma1}, we can assume that $\phi$ is surjective hence,
$c:=\dim (X)-\dim (Y)\geq 0$.
If $c\geq 2$,   any divisor on $X$ is contracted by $\phi$, so the theorem holds in this case.
There are thus two cases to consider, namely $c=0$ and $c=1$. We treat them separately.

\medskip
\noindent{\bf Case $c=0$.}  We have  $\dim (X)=\dim (Y)=n$. Let $Z\subset X$ be the (finite) union of the divisors
of $X$ contracted by $\phi$ and let $Z'\subset Y$ be the image of $Z$ in $Y$. By construction,
$\dim (Z')\leq n-2$. Set $Y':=Y\setminus Z$ and $X':=X\setminus \phi^{-1}(Z')$.
Since $Y$ is smooth, the map
$$\phi':=\phi\vert_{  X'}: X'\rightarrow Y' $$
is proper with finite fibers.

We now use the following Lefschetz-type result proved in \cite[Thm 2.1.4 and Remark 1.1.2, (iii)]{hammle}  and conjectured by Deligne.

\begin{theo}[Hamm-L\^e] \label{hammutile} Let $X_0$ be an $n$-dimensional smooth complex quasi-projective variety, let $\phi_0: X_0\rightarrow \P^N$ be
a morphism with finite fibers, and let $L\subset \P^N$ be a linear subspace of codimension $c$, in general position. Then the pair $(X_0,\phi_0^{-1}(L))$ is $(n-c)$-connected.
\end{theo}

We apply this statement to the composition of $\phi':X'\rightarrow Y'$ and an  embedding
of $Y$ into some $\P^N$. We   conclude that for a general
complete intersection surface $S_Y\subset Y$,
with inverse image
$S_X\subset X$, the pair $(X',S_X')$ is $2$-connected, where $S_X':=X'\cap S_X$.
It follows that the restriction map
$$H^2(X',\Q)\rightarrow H^2(S_X',\Q)$$
is injective.

Note that by Bertini's theorem, both surfaces $S_X$ and $S_Y$ are smooth.
Denoting by $\phi_S:S_X\rightarrow S_Y$ the restriction of $\phi$ to
$S_X$, we have, for any class $\alpha\in H^2(X,\R )$ such that
$\phi_* \alpha =0$ in $H^2(Y,\R )$,
\begin{equation}\label{for1}
(\phi_S)_*(\alpha\vert_{  S_X})=0\quad {\rm in}\ \  H^2(S_Y,\R ).
\end{equation}
Indeed, using the smoothness of $X$, $Y$, $S_X$, and $S_Y$, we easily show that
$$(\phi_*\alpha)\vert_{  S_Y}=(\phi_S)_*(\alpha\vert_{  S_X})\quad {\rm in}\ \ H^2(S_Y,\R ),$$
so (\ref{for1}) follows from the vanishing of $\phi_*\alpha$.

Assume furthermore that $\alpha$ is a pseudo-effective class on $X$. Then $\alpha\vert_{  S_X}$ is
a pseudo-effective class on  $S_X$ since $S_X$ is moving in $X$, and we just proved that it is annihilated by $\phi_{S*}$.
By Theorem \ref{curvecase}, we conclude that
the class $\alpha \vert_{  S_X}$ belongs to the real vector space spanned by classes of curves in $S_X$ contracted by
$\phi$, hence that the class
$\alpha \vert_{  S_X'}$ is equal to $0$, since $\phi_S$ is finite on $S'_X$. We thus proved that the class
$\alpha \vert_{  X'}\in H^2(X',\R )$ vanishes  in $H^2(S_X',\R )$.

Since the restriction map $H^2(X',\R )\rightarrow H^2(S_X',\R )$ is injective,
we conclude that $\alpha\vert_{  X'}=0$. We finally use the
following easy fact
(\cite[11.1.2]{voisinbook}).

\begin{lemm}\label{l1}
The kernel of the restriction
map
$H^2(X,\R )\rightarrow H^2(X',\R )$ is spanned by classes of divisors contained
in $X\setminus X'$.
\end{lemm}

By construction, these divisors are also  the divisors contracted by $\phi$ hence, this proves the theorem in the case $c=0$.

\medskip
\noindent{\bf Case $c=1$.} We now have $\dim (Y)=\dim (X)-1$ and the
general fibers $X_y:=\phi^{-1}(y)$ of
$\phi:X\rightarrow Y$ are smooth curves.
We first  observe the following.

\begin{lemm} Let  $\psi: \widetilde{X}\rightarrow Z$ be a smooth projective model
of the Stein factorization of $\phi$, with induced  birational morphism $\tau:\widetilde{X}\rightarrow X$. A   class $\alpha\in H^2(X,\R )$ is annihilated by
$\phi_*:H^2(X,\R )\rightarrow H^0(Y,\R )$ if and only if
the class $\tau^*\alpha$ is annihilated by $\psi_*:H^2(\widetilde{X},\R )\rightarrow H^0(Z,\R )$.
\end{lemm}

\begin{proof} Take a general point $y\in Y$. The restriction map
$H^0(Y,\R )\rightarrow H^0(\{y\},\R )$ is an isomorphism, and the composition
$$H^2(X,\R )\stackrel{\phi_*}{\rightarrow} H^0(Y,\R )\rightarrow H^0(\{y\},\R )$$
associates with $\alpha$ the sum $d$ of the degrees $d_i$ of its restrictions to the components
$X_{y,i}$, for $i\in\{1,\dots,m\}$, of the smooth curve $X_y:=\phi^{-1}(y)$.

Similarly, the composition
$$H^2(\widetilde  X,\R )\stackrel{\psi_*}{\rightarrow} H^0(Z,\R )\rightarrow H^0(\{z_i\},\R ),$$
where $z_i\in Z$ maps to $y$, associates with $\tau^*\alpha$ the degree $\widetilde d$ of its restriction to the smooth irreducible curve $\widetilde X_{z_i}$. We have $d=\sum_{i=1}^md_i=m\widetilde d$.

It follows that $d$ vanishes if and only if $\widetilde d$ does. This proves the lemma.
 \end{proof}

Note also that $\tau^*\alpha$ belongs to the real vector space
 spanned by classes of divisors contracted by $\psi$ if and only if
 $\alpha$ belongs to the real vector space
 spanned by classes of divisors contracted by $\phi$. Moreover, if $\alpha$ is pseudo-effective, so is $\tau^*\alpha$.
 Combining these arguments, we are reduced to the case where $\phi:X\rightarrow Y$ has
 connected fibers (and $X$ and $Y$ are still smooth).

Let $Y'\subset Y$ be the dense Zariski open subset  defined by the condition that $y\in Y'$ if and only if the fiber $X_y:=\phi^{-1}(y)$ is a smooth curve.
Set $X':=\phi^{-1}(Y')\subset X$.

Let $\alpha$ be a  pseudo-effective divisor class on $X$ which is annihilated by $\phi_*$.
Its restriction to $X'$ belongs to the kernel
of the map
$$H^2(X',\R )\rightarrow H^0(Y',R^2\phi'_*\R ),$$
which is the first quotient map
\begin{equation}\label{q}
H^2(X',\R )\rightarrow E_2^{0,2}
\end{equation}
of the Leray spectral sequence of $\phi'$. Recalling from
\cite{delignedeg} or \cite[II, 4.2.3]{voisinbook} that the Leray
spectral sequence of $\phi'$ degenerates at $E_2$, we find that the
kernel $L^1H^2(X',\R )$ of the map in (\ref{q})
 maps surjectively to
$E_2^{1,1}=H^1(Y',R^1\phi'_*\R )$.
The kernel of this last map is   $E_2^{2,0}=H^2(Y',R^0\phi'_*\R )$.

\begin{lemm} \label{leleray} Let $\alpha$ be as above. Then the restricted class $\alpha\vert_{  X'}$ maps to $0$ in
the Leray quotient
$H^1(Y',R^1\phi'_*\R )$.
\end{lemm}

\begin{proof}
Let  $C\subset Y$ be a curve which is a complete intersection of ample hypersurfaces in general position. Set $C':=C\cap Y'$, $X_C:=\phi^{-1}(C)$, and $X'_C:=X_C\cap X'$, and let $\phi'_C:X'_C\rightarrow C'$ be  the restriction of $\phi$.
According to Theorem \ref{hammutile}, the pair
$(Y',C')$ is $1$-connected hence, the restriction map
$$H^1(Y',R^1\phi'_*\R )\rightarrow H^1(C',R^1\phi'_{C*}\R )$$
is injective. It thus suffices to show that the image of $\alpha$ in
  $H^1(C',R^1\phi'_{C*}\R ) $ vanishes, which we do now.

The space $H^1(C',R^1\phi'_{C*}\R ) $ is the second Leray quotient of
the group $H^2(X'_C,\R )$. The image of $\alpha$ in this group can be obtained as well
by restricting first $\alpha$ to the smooth surface $X_C$, and observing that
the class $\alpha_C$ so obtained has the property that its restriction
to $X'_C$ belongs to the Leray level
$L^1H^2(X'_C,\R )$.
On the other hand, the class
$\alpha_C$ is pseudo-effective on $X_C$ and it is annihilated by the
map $\phi_{C*}:H^2(X_C,\R )\rightarrow H^0(C,\R )$. Hence we can apply
Theorem \ref{curvecase} to $\alpha_C$ and conclude that it belongs to the space spanned by
classes of components of fibers of $\phi_C:X_C\rightarrow C$.
It follows that the restriction
$\alpha_C\vert_{ X'_C}$ vanishes in the Leray quotient $H^1(C',R^1\phi'_{C*}\R )$.
\end{proof}

Lemma \ref{leleray} tells us  that
$\alpha\vert_{  X'}$ belongs to the
subspace
$$E_2^{2,0}=H^2(Y',R^0\phi'_*\R )\subset  H^2(X',\R ).$$
As we assumed that $\phi$ has connected fibers, $R^0\phi'_*\R =\R $
and $H^2(Y',R^0\phi'_*\R )\subset  H^2(X',\R )$ is nothing but
${\phi'}^*H^2(Y',\R )\subset H^2(X',\R )$.

We   have the following result.

\begin{lemm} \label{leproof} Let $\alpha\in N^1(X)_\R $.
Assume that its image $\beta\in H^2(X',\R )$ under the restriction map
  belongs
to ${\phi'}^*H^2(Y',\R )$.
Then $\beta$ belongs to
the image of the  composed map
$$N^1(Y)_\R \stackrel{{\phi}^*}{\rightarrow}H^2({X},\R )\rightarrow H^2(X',\R ).$$
More precisely, we have
\begin{equation}\label{decomp}\alpha={\phi}^*\beta'+\alpha',
\end{equation}
where $\beta'\in N^1(Y)_\R$ and $\alpha'$ is an $\R $-combination of classes
of divisors supported on $X\setminus X'$.
\end{lemm}

\begin{proof} We have $\beta=\alpha\vert_{  X'}$, with
$\alpha\in NS(X)_\R $ and $\beta={\phi'}^*\beta''$ for some
$\beta''\in H^2(Y',\R )$.
Let  $l $ be an ample class  on
$ X $, with degree $e >0$ on the general fibers of ${\phi}$. We have
by the projection formula
$$e\beta''=\phi'_*(l\cup {\phi'}^*\beta'')=\phi'_*(l\cup \beta),$$
which can be rewritten as
$$e\beta''=\phi_*(l\cup \alpha)\vert_{  Y'}.$$
The class $\beta':=\frac1e\phi_*(l\cup \alpha)$ belongs to
$N^1(Y)_\R $.
 We thus have  $$\beta={\phi'}^*\beta''=\phi^*\beta'\vert_{  X'} $$ for
some $\beta'\in NS( Y)_\R  $. Recalling that
$\beta=\alpha\vert_{  X'}$, we get
$$({\phi}^*\beta'-\alpha)\vert_{  X'}=0\,\,{\rm in}\,\, H^2(X',\R ).$$
By lemma \ref{l1}, the difference
$\alpha':={\phi}^*\beta'-\alpha$ belongs to the real vector subspace
of $N^1(X)_\R $ spanned by classes
of divisors supported on ${X}\setminus X'$.

Hence we proved that
$$\alpha={\phi}^*\beta'-\alpha',$$
for some $\beta'\in N^1(Y)_\R $ and some $\alpha'$ in
 the  real vector subspace
of $N^1(X)_\R $ spanned by classes
of divisors supported on ${X}\setminus X'$.
\end{proof}

This  lemma immediately implies the theorem when $c=1$. Indeed,  let
$\alpha$ be a degree-$2$ pseudo-effective class which is annihilated
by $\phi_*$. Using Lemmas \ref{leleray} and  \ref{leproof}, we
  decompose $\alpha$ as in (\ref{decomp}). Clearly,
  $\phi^*(N^1(Y)_\R)  $ is spanned by classes of
divisors contracted by $\phi$, and similarly,   the divisorial part
of the complement ${X}\setminus X'$, being sent by $\phi$ on the
proper closed algebraic subset $Y\setminus Y'\subset Y$, consists of
subvarieties contracted by $\phi$.
 \end{proof}

\subsection{The Strong   Conjecture \ref{conj2} for divisors}

 We now show that  for divisors, the strong form of Conjecture \ref{conj2}   follows  from its weak form.

\begin{theo}\label{divcone} The Strong Conjecture \ref{conj2}  holds true for divisors ($l=1$).
\end{theo}

\begin{proof}
As in the proof of Theorem \ref{theodiv}, we may by Proposition \ref{lemma1}   assume that $\phi$ is surjective
and that $c:=\dim(X)-\dim(Y)$ is 0 or 1. We denote by $\phi: X\xrightarrow{\psi} Z\xrightarrow{h} Y$ the Stein factorization of $\phi$, where $X$ and $Y$ are smooth and $Z$ is normal.

\medskip
\noindent{\bf Case $c=0$.}
The morphism $\phi$ is generically finite. Let $\alpha\in  H^2(X,\R) $ be a pseudo-effective cohomology class annihilated by $\phi_*:H^2(X,\R)\rightarrow H^2(Y,\R)$. By Theorem \ref{theodiv}, we can write $\alpha=\sum_ia_i[E_i]$, where the $a_i$  are real numbers and the $E_i$
are effective divisors contracted by $\phi$. In the Stein factorization, $\psi$ is birational and the $E_i$  are $\psi$-exceptional divisors. By a result of Lazarsfeld (\cite[Corollary 13]{KL}), the $\psi$-exceptional divisor $\sum_i a_iE_i$ is pseudo-effective if and only if $a_i\geq 0$ for all $i$. Hence $\alpha$ belongs to the convex cone spanned by classes of effective divisors of $X$ which are contracted by $\phi$.

\medskip\noindent{\bf Case $c=1$.}
We use the same reduction as in the proof of \cite[Theorem (1.12)]{mor}.
By the flattening theorem of Hironaka or Gruson-Raynaud, we can find a desingularization $g: Z'\rightarrow Z$ such the induced morphism $(X\times_ZZ')^{\textrm{main}}\rightarrow Z'$ is flat. Let $X'$ be the normalization of $(X\times_ZZ')^{\textrm{main}}$. We have a commutative diagram of induced morphisms
$$\xymatrix{
X'\ar[r]^f\ar[d]^{\psi'}& X\ar[d]^{\psi}\ar[dr]^{\phi}\\
Z'\ar[r]^g & Z\ar[r]^h& Y.
}
$$
Let $\alpha\in  H^{2}(X,\R) $ be a pseudo-effective cohomology class annihilated by $\phi_*:H^2(X,\R)\to H^2(Y,\R)$.  
 By Theorem \ref{theodiv}, $\alpha$ belongs to the real vector space spanned by classes of hypersurfaces contracted by $\phi$, hence also by $\psi$. Since $X$ is smooth, $\alpha$ is in particular the class of an $\R$-Cartier $\R$-divisor on $X$, hence
   $f^*\alpha$ is the class of an $\R$-Cartier $\R$-divisor
$D:=\sum_ia_i D_i $, 
where $a_i\in\R$ and the $D_i$ are irreducible hypersurfaces on $X'$ contracted by $\psi'$ (\ie, $\psi'(D_i)\ne Z'$).

 More precisely, since $X'\rightarrow (X\times_ZZ')^{\rm main}$ is finite and $(X\times_ZZ')^{\rm main}\rightarrow Z'$ is flat, the $\psi'(D_i)$ are   hypersurfaces in $Z'$. For each irreducible component $E_j$ of $\psi'(D)$, we write $\psi'^{*} E_j =\sum_ib^j_iD_i$. Set
$$c_j:=\min_{ b^j_i\ne 0} \frac{a_i}{b^j_i} $$
and $E:=\sum_jc_jE_j$; since $Z'$ is smooth, this an $\R$-Cartier $\R$-divisor. The $\R$-Cartier $\R$-divisor $G:=D-\psi^{\prime *}E$ is then  contracted by $\psi'$. 
 Moreover, for any irreducible component $E_j$ of $\psi'(G)$, there exists an irreducible component  of $\psi'^{-1}(E_j)$ not contained in $ \Supp (G)$.

To conclude the proof of Theorem \ref{divcone}, we just need to prove the following.

\medskip\noindent{\bf Claim.} The class $[E]$  in $H^2(Z',\R)$ is pseudo-effective.

Indeed, if the claim holds,   there exists a sequence $(V_n)_n$ of effective divisors   of $Z'$ such that $ \lim_{n\to \infty}[V_n]=   [E]$ in $H^2(Z', \R)$. We have then
$$\alpha=f_*f^*\alpha=[f_*D]=[f_*G]+[f_*\psi^{\prime *}E]=[f_*G]+\lim_{n\to \infty}[f_*\psi^{\prime *}V_n]
$$
hence $\alpha$ is pseudo-effective.

\medskip\noindent{\bf Proof of the claim.}
 Let $\eps: \widetilde{X}\rightarrow X'$ be a desingularization and set $\varphi:=\psi'\circ \eps:\widetilde{X}\to Z'$. We have
 $\eps^*D=\eps^*G+\varphi^*E$ and $\eps^*G$ is an effective $\R$-Cartier $\R$-divisor contracted by $\varphi$. Moreover, for any prime divisor
 $P$ on $Z'$, there exists a prime divisor $\Gamma $ on $ \widetilde{X}$ such that $\varphi(\Gamma)=P$ and $\Gamma$ is not contained in $\Supp (\eps^*G)$.
 In other words, $\eps^*G$ is   a $\varphi$-degenerate divisor in the sense of \cite[Definition 2.14]{GL}.

 We now consider the $\sigma$-decomposition (see for instance  \cite[Definition 2.8 and Lemma 2.9]{GL}) $$\eps^*D=P_{\sigma}(\eps^*D)+N_{\sigma}(\eps^*D).$$
Recall that $P_{\sigma}(\eps^*D)$ is pseudo-effective and 
$$N_{\sigma}(\eps^*D):=\sum_{\Gamma}\sigma_\Gamma(D)\Gamma$$ is an effective $\R$-divisor, where, fixing some ample divisor $A$ on $\widetilde{X}$, we set
$$\sigma_\Gamma(D):=\min_{ t>0} \{\mult_{\Gamma}(D')\mid D'\sim_\Q D+t A \}.$$
 Let $H$ be a very ample divisor on $Z'$. Then by definition of $N_{\sigma}$, we have
 $$N_{\sigma}(\eps^*D+\varphi^*H)\leq N_{\sigma}(\eps^*D).$$
 On the other hand, $\eps^*D+\varphi^*H=\varphi^*(E+H)+\eps^*G$. 
  We may assume that $E+H $ is effective. Then by \cite[Lemma 2.16]{GL}, we have
 $$\eps^*G\leq N_{\sigma}(\eps^*D+\varphi^*H).$$
 Thus $\eps^*G\leq N_{\sigma}(\eps^*D)$ and hence 
 $$\varphi^*E=\eps^*D-\eps^*G=P_{\sigma}(\eps^*D)+(N_{\sigma}(\eps^*D)-\eps^*G)$$
  is pseudo-effective. Thus $E$ is pseudo-effective.
\end{proof}

\section{Weakly pseudo-effective cycles} \label{secpsefan}

We turn now to Conjecture \ref{conj5}, a variant of Conjecture  \ref{conj1}  involving weakly pseudo-effective classes. This variant is quite natural, as we already used
currents in the arguments in   previous sections and they seem to be
the right tool to attack the problem. We will show however  that it would have strong consequences on
the generalized Hodge conjecture.

Let $X$ be a {\em smooth projective variety such that $H^{i,0}(X)=0$  for $i>0$.} Equivalently, the Hodge coniveau
of $H^i(X,\Q)$ is $\geq 1$ for any $i>0$.
The following conjecture  is  a particular case of the generalized Hodge conjecture of  Grothendieck (\cite{gro}).

 \begin{conj}[Grothendieck] \label{conjGHC} If $X$ is as above, there exists a proper closed algebraic subset
 $Y\subset X$ such that the restriction map $H^i(X,\Q)\rightarrow H^i(X\setminus Y,\Q)$
 is zero for any $i>0$.
 \end{conj}

If $n=\dim (X)$, by the K\"{u}nneth decomposition, the conclusion is equivalent to the fact that the cycle class
 $$[\Delta_X- x\times X]\in H^{2n}(X\times X,\Q) $$
  vanishes in $H^{2n}(X\times (X\setminus Y),\Q)$  or, equivalently, comes from
 a   class in $H_{2n}(X\times Y,\Q)$.

 A stronger version of Conjecture \ref{conjGHC} asserts that there exist a proper
 closed algebraic subset $Y\subset X$ and  an $n$-cycle
 $Z\in \cZ_n(X\times X)_\Q$ supported on $X\times Y$ such that
 the class of $Z$ in $H^{2n}(X\times X,\Q)$ is equal to $[\Delta_X- x\times X]$. We prove below (Theorem \ref{th65}) that {\em this strong form of Conjecture \ref{conjGHC} is a consequence of Conjecture \ref{conj5}} made in the introduction.

 We begin with the following.

 \begin{prop} \label{lewepsef} Let $X$ be as above (that is $H^{i,0}(X)=0$ for $i>0$) and let $\omega \in H^2(X,\Z)$ be
 the class of an ample line bundle on $X$.
 Then for $\lambda\gg 0$, the class
 $$\lambda(\pr_1^*\omega+\pr_2^*\omega)^{n-1}\cup \pr_2^*\omega+[\Delta_X- x\times X]$$
 is weakly pseudo-effective, where $\pr_1$ and $\pr_2$ are the two projections from
 $X\times X$ to $X$.
 \end{prop}

 \begin{proof} Let $(\alpha_i^{p,q})_{i,\, p+q>0}$ be a basis for $H^{ >0}(X,\C)$, where $\alpha_i^{p,q}$ has type $(p,q)$, and let
   $(\beta_i^{n-p,n-q})$ be the dual basis for  $H^{<2n}(X,\C)$. By Hodge theory,   these
 classes are represented by  closed forms $\tilde{\alpha}_i^{p,q}$ and $\tilde{\beta}_i^{n-p,n-q}$ of the given type.
 The class $[\Delta_X- x\times X]$ is the cohomology class
 of the closed form
 \begin{eqnarray}\label{eqdiag}\sum_{p,q,i}\pr_1^*\tilde{\beta}_i^{n-p,n-q}\wedge \pr_2^*\tilde{\alpha}_i^{p,q},
 \end{eqnarray}
 where only   couples $(p,q)$ with $p+q>0$ appear. Our assumption is
 $H^{i,0}(X)=0$ for $ i>0$, which by Hodge symmetry also gives $H^{0,i}(X)=0$, so  only
 couples $(p,q)$ with $p>0$ and $q>0$ actually appear.
 The form appearing in (\ref{eqdiag}) is thus a sum of forms
 $$ \Re (\pr_1^*\tilde{\beta}^{n-p,n-q}\wedge \pr_2^*\tilde{\alpha}^{p,q}),$$
 with $p>0$ and $q>0$.

 In local coordinates, every form decomposes as a sum of decomposable forms, and
 the fact that $p>0$ and $q>0$ implies that each term in the decomposition is of the form
 $$\tilde{\alpha}^{p,q}=dz_i\wedge d\overline{z}_j\wedge {\alpha'}^{p-1,q-1}$$
 for some indices $i$ and $j$ and some form ${\alpha'}^{p-1,q-1}$.
 Proposition \ref{lewepsef}   thus follows from the following lemma.\end{proof}

 \begin{lemm} Let $\omega$ be a positive $(1,1)$-form on an open subset $U$ of
 $\C^n$ and let $\tilde{\beta}^{n-p,n-q}$ and ${\alpha'}^{p-1,q-1}$ be, respectively, an $(n-p,n-q)$-form and a $(p-1,q-1)$-form on
 $U$. Set
 $$\gamma:=\Re (\pr_1^*(\tilde{\beta}^{n-p,n-q})\wedge \pr_2^*( dz_i\wedge d\overline{z}_j\wedge {\alpha'}^{p-1,q-1})).$$
 Then there exists locally on $U\times U$ a positive real number $\lambda$ such that
 \begin{equation}\label{ineggamma}
 \lambda(\pr_1^*\omega+\pr_2^*\omega)^{n-1}\wedge \pr_2^*\omega \geq \gamma
 \end{equation}
 as real $(n,n)$-forms on $U\times U$.
 \end{lemm}

\begin{proof}  For any $(n-1,n-1)$-form $\beta^{n-1,n-1}$ and any $(1,1)$-form
$\beta^{1,1}$, we write
$$\Re (\beta^{n-1,n-1}\wedge \beta^{1,1})=\Re \beta^{n-1,n-1}\wedge \Re \beta^{1,1}-\Im \beta^{n-1,n-1}\wedge \Im \beta^{1,1}.$$
Applying this to $\beta^{n-1,n-1}= \pr_1^*(\tilde{\beta}^{n-p,n-q})\wedge \pr_2^*( {\alpha'}^{p-1,q-1})$
and $\beta^{1,1}=\pr_2^*( dz_i\wedge d\overline{z}_j)$ gives us
\begin{eqnarray}\label{gamma}
\gamma&=&\Re (\pr_1^*(\tilde{\beta}^{n-p,n-q})\wedge \pr_2^*( {\alpha'}^{p-1,q-1}))\wedge \pr_2^*(\Re (dz_i\wedge d\overline{z}_j))\\
&&\nonumber
-\Im (\pr_1^*(\tilde{\beta}^{n-p,n-q})\wedge \pr_2^*( {\alpha'}^{p-1,q-1}))\wedge \pr_2^*(\Im  (dz_i\wedge d\overline{z}_j)).
\end{eqnarray}
The forms $\Re (\pr_1^*(\tilde{\beta}^{n-p,n-q})\wedge \pr_2^*( {\alpha'}^{p-1,q-1}))$ and
$\Im (\pr_1^*(\tilde{\beta}^{n-p,n-q})\wedge \pr_2^*( {\alpha'}^{p-1,q-1}))$ are real of type $(n-1,n-1)$ on
$U\times U$, hence satisfy local
 inequalities
\begin{eqnarray*}
\Re (\pr_1^*(\tilde{\beta}^{n-p,n-q})\wedge \pr_2^*( {\alpha'}^{p-1,q-1}))&\leq &\lambda'(\pr_1^*\omega+\pr_2^*\omega)^{n-1},\\
 \Im (\pr_1^*(\tilde{\beta}^{n-p,n-q})\wedge \pr_2^*( {\alpha'}^{p-1,q-1}))&\leq &\lambda'(\pr_1^*\omega+\pr_2^*\omega)^{n-1}
\end{eqnarray*}
for   $\lambda'\gg0$.
The forms $\Re (dz_i\wedge d\overline{z}_j)$ and $\Im  (dz_i\wedge d\overline{z}_j)$ are real
of type $(1,1)$ on $U$, hence satisfy
local inequalities
$$\Re (dz_i\wedge d\overline{z}_j)\leq \lambda''\omega\quad {\rm and} \quad\Im  (dz_i\wedge d\overline{z}_j)\leq\lambda''\omega$$
for  $\lambda''\gg0$. Letting $\lambda= 2 \lambda'\lambda''$, we   get from
(\ref{gamma}) the desired inequality (\ref{ineggamma}).
\end{proof}

\begin{rema}{\rm Proposition \ref{lewepsef} suggests the following question: under the same assumptions, is the class
$$\lambda(\pr_1^*\omega+\pr_2^*\omega)^{n-1}\cup \pr_2^*\omega+[\Delta_X- x\times X]$$
pseudo-effective for $\lambda\gg0$ ?}
\end{rema}

 As an immediate consequence of Proposition \ref{lewepsef}, we get the following.

 \begin{theo}\label{th65} If Conjecture \ref{conj5} is true then, for any smooth projective variety $X$  such that
 $H^{i,0}(X)=0$ for $i>0$,  there exist a proper
 closed algebraic subset $Y\subset X$ and  an $n$-cycle
 $Z\in \cZ_n(X\times X)_\Q$ supported on $X\times Y$, where $n:=\dim(X)$,  such that
 the class of $Z$ in $H^{2n}(X\times X,\Q)$ is equal to $[\Delta_X- x\times X]$.
 \end{theo}
 \begin{proof} Let $\omega$ be the class of an ample divisor on $X$.
By Proposition \ref{lewepsef}, we know that for $\lambda\gg 0$,
 the class
 $$A:=\lambda(\pr_1^*\omega+\pr_2^*\omega)^{n-1}\wedge \pr_2^*\omega+[\Delta_X- x\times X]\in H^{2n}(X\times X,\Q)$$
 is weakly pseudo-effective.
 On the other hand, the class $A$ is clearly annihilated by $\pr_{2*}:H^{2n}(X\times X,\Q)\rightarrow H^0(X,\Q)$.
 Conjecture \ref{conj5} tells us that there exist $n$-dimensional  subvarieties
 $Y_i\subset X\times X$ contracted by $\pr_2$ such that
 $$A=\sum_i\alpha_i[Y_i],$$
  for some  $\alpha_i\in \R$.
 As $A$ is rational, this immediately implies that
 \begin{equation}\label{eqratcoeff}
 A=\sum_i\beta_i[Y_i]
 \end{equation}
  for some $\beta_i\in \Q$.
 Since  the $Y_i$ are contracted, their images in $X$ are proper closed algebraic subsets, and there exists
 a proper closed algebraic subset
 $Y\subset X$, such that $\pr_2(Y_i)\subset Y$ for all $i$. Enlarging $Y$ if necessary, we may furthermore assume
 that $Y$ is a divisor whose class is a multiple of $\omega$.
 Formula (\ref{eqratcoeff}) then says that $A$ is the class
 of a $\Q$-cycle supported on $X\times Y$.
 On the other hand, the class $(\omega_1+\omega_2)^{n-1}\cup \omega_2$ is also clearly
 the class
 of a $\Q$-cycle supported on $X\times Y$. Hence we conclude that
 $[\Delta_X- x\times X]$ is the class
 of a $\Q$-cycle supported on $X\times Y$.
 \end{proof}

 \end{document}